\documentclass[12pt]{amsart}
\usepackage{verbatim}
\usepackage{tikz}
\usepackage{bbm}
\usepackage{hyperref}
\title[Chromatic numbers of iterated arc graphs]
{Iterated arc graphs}
\author[Rorabaugh]{Danny Rorabaugh}
\address[D. Rorabaugh, C. Tardif, and D. Wehlau]{Department of Mathematics and Statistics,
Queen's University, Kingston}
\author[Tardif]{Claude Tardif}
\address[C. Tardif, D. Wehlau, and I. Zaguia]{Department of Mathematics and Computer Science,
Royal Military College of Canada, Kingston}
\author[Wehlau]{David Wehlau}
\author[Zaguia]{Imed Zaguia}
\thanks{Research supported by grants from NSERC and CDARP}
\newtheorem{proposition}{Proposition}
\newtheorem{theorem}[proposition]{Theorem}
\newtheorem{lemma}[proposition]{Lemma}
\newtheorem{corollary}[proposition]{Corollary}

\newfont{\Bb}{msbm10 scaled\magstep1}

\newcommand{\binomial}[2]{\genfrac{(}{)}{0pt}{}{#1}{#2}}
\newcommand{\nd}{\mathcal{N}}

\begin{document}

\begin{abstract} 
The arc graph $\delta(G)$ of a digraph $G$ is the digraph with
the set of arcs of $G$ as vertex-set, where the arcs of
$\delta(G)$ join consecutive arcs of $G$.
In 1981, Poljak and R\"{o}dl characterised the chromatic number
of $\delta(G)$ in terms of the chromatic number of $G$ when
$G$ is symmetric (i.e., undirected).
In contrast, directed graphs with equal chromatic numbers
can have arc graphs with distinct chromatic numbers.
Even though the arc graph of a symmetric graph is not symmetric,
we show that the chromatic number of the iterated arc graph
$\delta^k(G)$ still only depends on the chromatic number of $G$
when $G$ is symmetric.
\end{abstract}

\maketitle

{\small \noindent{\bf Keywords: Arc graphs, chromatic numbers, free distributive lattices, Dedekind numbers}
}
\smallskip

\section{Introduction}
The arc graph $\delta(G)$ of a digraph $G$ is the digraph with
the set $A(G)$ of arcs of $G$ as vertex-set, where the arcs of 
$\delta(G)$ join consecutive arcs of $G$. The iterated arc graphs
$\delta^{k}(G)$, $k \geq 1$ are defined recursively by
$\delta^{k}(G) = \delta(\delta^{k-1}(G))$. However it is possible 
to interpret $\delta^k(G)$ in terms of sequences of vertices of $G$:
\begin{eqnarray*}
V(\delta^{k}(G)) & = & \{ (u_1, \ldots, u_k) \in V(G)^k \mid \\&&\quad 
(u_i,u_{i+1}) \in A(G) \mbox{ for } i = 1, \ldots, k-1 \}, \\
A(\delta^{k}(G)) & = & \{ ((u_0, \ldots, u_{k-1}),(u_1, \ldots, u_k)) \mid \\&& \quad
(u_0, \ldots, u_{k-1}), (u_1, \ldots, u_k) \in V(\delta^{k}(G)) \}.
\end{eqnarray*}
In particular, the iterated arc graphs of complete graphs with loops are
the well-known de Bruijn graphs. The iterated arc graphs of transitive 
tournaments are a folklore construction of graphs with large chromatic 
numbers and no short odd cycles (see~\cite{hellnesetril}). 

We will be investigating chromatic numbers of iterated arc graphs.
Here, the chromatic number of a digraph is defined as the minimum
number of colours needed to colour its vertices so that the 
endpoints of an arc have different colours. Thus, the direction
of an arc has no effect on the chromatic number. 
In~\cite{entringerharner}, Entriger and Harner give the following 
relations between the chromatic number of a digraph and that of
its arc graph. 
\begin{theorem}[\cite{entringerharner}] \label{ea} \mbox{}
\begin{itemize}
\item[(i)] If $\chi(\delta(G)) \leq n$, then $\chi(G) \leq 2^n$.
\item[(ii)] If $\chi(G) \leq \binomial{n}{\lfloor n/2 \rfloor}$,
then $\chi(\delta(G)) \leq n$.
\end{itemize}
\end{theorem}
Inductively, the bound gives $\theta\!\left(\log^{(k)}(\chi(G))\!\right)$ behaviour
for $\chi(\delta^k(G))$ in terms of $\chi(G)$.

Undirected graphs can be viewed as symmetric digraphs
with each edge corresponding to an opposite pair of arcs.
In~\cite{poljakrodl}, Poljak and R\"{o}dl give 
a characterisation of the chromatic number of the arc graph of an
undirected graph.
\begin{theorem}[\cite{poljakrodl}] \label{pr}
If $G$ is an undirected graph, then
$$
\chi(\delta(G)) = 
\min \left \{ n \mid \chi(G) \leq \binomial{n}{\lfloor n/2 \rfloor} \right \}.
$$
\end{theorem}
In particular, if $G$ is undirected, then $\chi(\delta(G))$ depends on 
$\chi(G)$ alone and not on the structure of $G$. In contrast, directed
graphs with equal chromatic numbers can have arc graphs with distinct
chromatic numbers. For instance, let $C$ be the cyclic tournament on
three vertices, and $T$ the transitive tournament on three vertices.
then $\chi(C) = \chi(T) = 3$, while $\chi(\delta(C)) = 3$ and $\chi(\delta(T)) = 2$.

Now, what about iterated arc graphs of undirected graphs? 
If $G$ is undirected, then Theorem~\ref{pr} characterises
$\chi(\delta(G))$ in terms of $\chi(G)$, but since $\delta(G)$
is directed,  Theorem~\ref{pr} cannot be used to
characterise $\chi(\delta^2(G))$ in terms of $\chi(\delta(G))$.

In this paper, we show that $\chi(\delta^k(G))$ is indeed characterised
by $\chi(G)$. The characterisation uses numbers defined in terms
of specific posets. Let $\overline{K}_n$ be the antichain
of size $n$. For a poset $P$, let $\mathcal{I}(P)$ be the poset of ideals
(i.e., downsets) of $P$, ordered by inclusion. Then $\mathcal{I}(P)$
is itself a poset, and the construction can be iterated.
We let $b(n,k)$ be the maximum size of an antichain in $\mathcal{I}^k(\overline{K}_n)$
(i.e., the {\em width} of $\mathcal{I}^k(\overline{K}_n)$).
Our main result is the following.

\begin{theorem} \label{rtwz} For any undirected graph $G$ and any integer $k \geq 1$, 
$$
\chi(\delta^k(G)) = 
\min \left \{ n \mid \chi(G) \leq b(n,k) \right \}.
$$
\end{theorem}
In particular, $\mathcal{I}(\overline{K}_n)$ is the boolean lattice with
$n$ generators, and $b(n,1) = \binomial{n}{\lfloor n/2 \rfloor}$
by Sperner's theorem.
Thus the case $k=1$ of Theorem~\ref{rtwz} is Theorem~\ref{pr}.
$\mathcal{I}^2(\overline{K}_n)$ is the free distributive lattice
with $n$ generators. The largest known antichain in
$\mathcal{I}^2(\overline{K}_n)$ consists of the ideals
with exactly $2^{n-1}$ elements, that is, half
the elements of $\mathcal{I}(\overline{K}_n)$.
In~\cite{mchard}, it has been verified that for $n$ up to $6$,
this is indeed the unique maximum antichain in $\mathcal{I}^2(\overline{K}_n)$.
Thus the known values of $b(n,2)$ are as follows.
$$
b(3,2) = 4, b(4,2) = 24, b(5,2) = 621, b(6,2) = 492288.
$$
In~\cite{mchard}, McHard conjectures that the ideals
with exactly half the elements of $\mathcal{I}(\overline{K}_n)$
always constitute the unique maximum antichain of $\mathcal{I}^2(\overline{K}_n)$.
Note that the ideals of a given size in $\mathcal{I}^{k-1}(\overline{K}_n)$
constitute a level in $\mathcal{I}^k(\overline{K}_n)$.
The posets in which some level constitutes a maximum antichain
are called {\em Sperner posets}. Thus McHard's conjecture implies that
the free distributive lattices are Sperner posets. The conjecture
seems plausible, though highly difficult. Even the cardinalities
of free distributive lattices, known as the Dedekind numbers, are not known 
in closed form. In \cite{markowsky}, it is shown that for any fixed $\ell$
there exists a polynomial $p_{\ell}$ such that $p_{\ell}(n)$ is the number
of elements at the $\ell$-th level of $\mathcal{I}^2(\overline{K}_n)$.

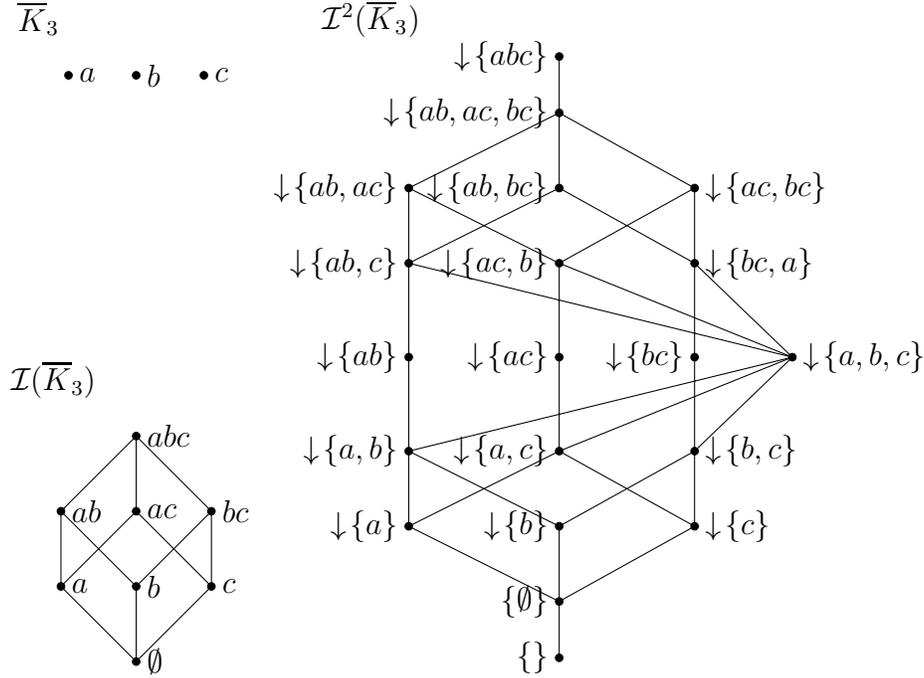
\begin{figure}
\begin{tikzpicture}
	\begin{scope}
		\filldraw (-.9,0) node[right]{$a$} circle(.05);
		\filldraw (0,0) node[right]{$b$} circle(.05);
		\filldraw (.9,0) node[right]{$c$} circle(.05);
		\draw (-1.3,.75) node{$\overline{K}_3$};
	\end{scope}
	
	\begin{scope}[yshift=-7.8cm]
		\filldraw (0,0) node[right]{$\emptyset$} circle(.05);
		\filldraw (-1,1) node[right]{$a$} circle(.05);
		\filldraw (0,1) node[right]{$b$} circle(.05);
		\filldraw (1,1) node[right]{$c$} circle(.05);
		\filldraw (-1,2) node[right]{$ab$} circle(.05);
		\filldraw (0,2) node[right]{$ac$} circle(.05);
		\filldraw (1,2) node[right]{$bc$} circle(.05);
		\filldraw (0,3) node[right]{$abc$} circle(.05);
		\draw (-1.1,3.75) node{$\mathcal{I}(\overline{K}_3)$};
		\draw (0,0)--(-1,1)--(-1,2)--(0,3)--(0,2)--(-1,1);
		\draw (-1,2)--(0,1)--(0,0)--(1,1)--(0,2);
		\draw (0,1)--(1,2)--(0,3);
		\draw (1,1)--(1,2);
	\end{scope}
	
	\begin{scope}[xshift=160, yshift=-7cm]
		\filldraw (0,-.75) node[left]{$\{\}$} circle(.05);
		\filldraw (0,0) node[left]{$\{\emptyset\}$} circle(.05);
		\filldraw (-2,1) node[left]{$\downarrow\!\{a\}$} circle(.05);
		\filldraw (0,1) node[left]{$\downarrow\!\{b\}$} circle(.05);
		\filldraw (1.8,1) node[right]{$\downarrow\!\{c\}$} circle(.05);
		\filldraw (-2,2) node[left]{$\downarrow\!\{a,b\}$} circle(.05);
		\filldraw (0,2) node[left]{$\downarrow\!\{a,c\}$} circle(.05);
		\filldraw (1.8,2) node[right]{$\downarrow\!\{b,c\}$} circle(.05);
		\filldraw (-2,3.25) node[left]{$\downarrow\!\{ab\}$} circle(.05);
		\filldraw (0,3.25) node[left]{$\downarrow\!\{ac\}$} circle(.05);
		\filldraw (1.8,3.25) node[left]{$\downarrow\!\{bc\}$} circle(.05);
		\filldraw (3.1,3.25) node[right]{$\downarrow\!\{a,b,c\}$} circle(.05);
		\filldraw (-2,4.5) node[left]{$\downarrow\!\{ab,c\}$} circle(.05);
		\filldraw (0,4.5) node[left]{$\downarrow\!\{ac,b\}$} circle(.05);
		\filldraw (1.8,4.5) node[right]{$\downarrow\!\{bc,a\}$} circle(.05);
		\filldraw (-2,5.5) node[left]{$\downarrow\!\{ab,ac\}$} circle(.05);
		\filldraw (0,5.5) node[left]{$\downarrow\!\{ab,bc\}$} circle(.05);
		\filldraw (1.8,5.5) node[right]{$\downarrow\!\{ac,bc\}$} circle(.05);
		\filldraw (0,6.5) node[left]{$\downarrow\!\{ab,ac,bc\}$} circle(.05);
		\filldraw (0,7.25) node[left]{$\downarrow\!\{abc\}$} circle(.05);
		\draw (-2.5,7.75) node{$\mathcal{I}^2(\overline{K}_3)$};
		\draw (0,-.75)--(0,0)--(-2,1)--(-2,2)--(-2,3.25)--(-2,4.5)--(-2,5.5)--(0,6.5)--(0,7.25);
		\draw (0,0)--(0,1)--(-2,2)--(3.1,3.25)--(-2,4.5)--(0,5.5)--(0,6.5)--(1.8,5.5)--(0,4.5)--(0,3.25)--(0,2)--(1.8,1)--(1.8,2)--(1.8,3.25)--(1.8,4.5)--(0,5.5);
		\draw (0,0)--(1.8,1);
		\draw (-2,1)--(0,2)--(3.1,3.25)--(0,4.5)--(-2,5.5);
		\draw (0,1)--(1.8,2)--(3.1,3.25)--(1.8,4.5)--(1.8,5.5);
	\end{scope}

\end{tikzpicture}

\caption{$\mathcal{I}^k(\overline{K}_3)$, $k \in \{0,1,2\}.$} \label{fig:K3}
\end{figure}

By inspecting the 20-element $\mathcal{I}^2(\overline{K}_3)$--see Figure~\ref{fig:K3}--it 
is easy to check that $b(3,3)=7$. The levels $8$, $9$, $11$, $12$ 
 of $\mathcal{I}^3(\overline{K}_3)$ each
constitute antichains of size $7$, but level $10$ is an antichain
of size $6$. Thus the level of maximum size is not necessarily the
middle level. Apart from the trivial cases $b(1,k) = 1$ and $b(2,k) = 2$ 
for all $k$, no further values $b(n,k)$ are known.

\section{Proof of Theorem \ref{rtwz}}
We will show that an $n$-colouring of $\delta^k(G)$ corresponds 
to a homomorphism from $G$ to a suitably defined digraph
$\nd(\mathcal{I}^k(\overline{K}_n))$. When $G$ is undirected, 
its edges must be
mapped to the symmetric arcs of $\nd(\mathcal{I}^k(\overline{K}_n))$. 
These symmetric arcs span a graph which retracts to $K_{b(n,k)}$.
The details of this argument are provided below.

\subsection{The right adjoint of the arc graph construction}

Viewed as a digraph functor, $\delta$ admits a kind of ``right adjoint''.
More precisely, there is a construction $\delta_R$ such that 
there exists a homomorphism from $\delta(G)$ to $K$ if and only if 
there exists a homomorphism from $G$ to $\delta_R(K)$. (Note that our 
definition of a right adjoint is less restrictive than the standard categorial 
definition, in which a correspondence between morphisms is required.)
Here, a {\em homomorphism} 
$\phi: G \rightarrow H$ is a map $\phi$ from the vertex set of $G$ to that of $H$ 
such that if $(u,v)$ is an arc of $G$, then $(\phi(u),\phi(v))$ is an arc of $H$.

For a digraph $K$, $\delta_R(K)$ is the digraph defined as follows.
\begin{itemize}
\item The vertices of $\delta_R(K)$ are the ordered pairs $(X,Y)$ such that 
$X$ and $Y$ are sets of vertices of $K$ with an arc $(x,y)$ between
all vertices $x$ of $X$ and all vertices $y$ of $Y$.
\item The arcs of $\delta_R(K)$ are ordered pairs $((X,Y),(Z,W))$ such that
$Y \cap Z \neq \emptyset$.
\end{itemize}
The sets $X, Y$ of vertices of $K$ used in the definition of $\delta_R(K)$
are allowed to be empty. In particular, if $K_0$ is the graph with no vertex and
no edge, then $\delta_R(K_0)$ is a single vertex, and $\delta_R^2(K_0)$ has three
vertices and one arc.

We use the following result.
\begin{lemma}[\cite{fonioktardif}]
Given two digraphs $G$ and $K$, there exists a homomorphism
of $\delta(G)$ to $K$ if and only if there exists a homomorphism
of $G$ to $\delta_R(K)$.
\end{lemma}
\begin{proof}
We include the sketch of an elementary proof to make the paper self contained.
First note that a homomorphism $\phi: G \rightarrow H$ induces homomorphisms
$\delta(\phi): \delta(G) \rightarrow \delta(H)$ defined by 
$\delta(\phi)(u,v) = (\phi(u), \phi(v))$, and
$\delta_R(\phi): \delta_R(G) \rightarrow \delta_R(H)$ defined by 
$\delta_R(\phi)(X,Y) = (\phi(X),\phi(Y))$. Second, note that
there are homomorphisms from $\delta(\delta_R(G))$ to $G$ defined
by mapping $((X,Y),(Z,W))$ to any element of $Y \cap Z$,
and from $G$ to $\delta_R(\delta(G))$ defined by
mapping $u$ to $(u^-,u^+)$, where $u^-$ and $u^+$ are respectively
the sets of arcs entering and leaving $u$.

Therefore, if  there exists a homomorphism from $\delta(G)$ to $K$
then there exists a homomorphism from $\delta_R(\delta(G))$ to $\delta_R(K)$,
which composed with a homomorphism from $G$ to $\delta_R(\delta(G))$
yields a homomorphism from $G$ to $\delta_R(K)$. Similarly,
if  there exists a homomorphism from $G$ to $\delta_R(K)$,
then there exists a homomorphism from $\delta(G)$ to $\delta(\delta_R(K))$,
and the latter admits a homomorphism to $K$.
\end{proof}
 
\begin{corollary} \label{coltodr}
Given two digraphs $G$ and $K$ and any $k \geq 1$, there exists a homomorphism
of $\delta^k(G)$ to $K$ if and only if there exists a homomorphism
of $G$ to $\delta_R^k(K)$.
\end{corollary}

A subdigraph $H$ of a digraph $G$ is called a {\em retract} of $G$
if there exists a homomorphism $\rho: G \rightarrow H$ such that the 
restriction of $\rho$ to $H$ is the identity.
\begin{lemma} \label{retract}
The digraph $\delta_R(K)$ retracts to its subdigraph induced by
the vertices $(X,Y)$ such that 
$$
Y = \{ y \in V(K) | (x,y) \in A(K) \mbox{ for all } x \in X \}
$$
and 
$$
X = \{ x \in V(K) | (x,y) \in A(K) \mbox{ for all } y \in Y \}.
$$
\end{lemma}
\begin{proof}
Let $\rho: \delta_R(K) \rightarrow \delta_R(K)$ be the map
defined by $\rho(X,Y) = (X',Y')$, where $X'$ is the set of common inneighbours
of $Y$, and $X'$ is the set of common outneighbours of $Y'$.
Then $\rho$ is easily seen to be a retraction on the prescribed
subdigraph.
\end{proof}

\subsection{Nondomination digraphs of posets}
The {\em nondomination digraph} $\nd(P)$ of a poset $P$ is the 
digraph which has the elements of $P$ for vertices, and for arcs
the ordered pairs $(u,v)$ such that $u$ is strictly less than
$v$ or $u$ and $v$ are incomparable. In other words,
if $G = \nd(P)$, then $A(G)$ is the complement in $V(G)^2$ 
of the relation $\geq$.
Note that $K_n = \nd(\overline{K}_n)$, where  $\overline{K}_n$ is 
the antichain of size $n$. The constructions,
$\delta_R$, and $\nd$ connect to the construction $\mathcal{I}$
of the Introduction as follows.
\begin{lemma}
For any poset $P$, $\delta_R(\nd(P))$ retracts to 
$\nd(\mathcal{I}(P))$.
\end{lemma}
\begin{proof} By Lemma~\ref{retract}, $\delta_R(\nd(P)))$ retracts
to its subdigraph $G$ induced by the vertices $(X,Y)$ such that $Y$ is the common
outneighbourhood of all vertices in $X$ and $X$ is the common
inneighbourhood of all vertices in $Y$. We will show that $G = \nd(\mathcal{I}(P))$.
 
For $(X,Y) \in V(G)$, suppose that there exists a vertex $u$ of $\nd(P)$ not contained 
in $X$ or $Y$. Since $u \not\in X$, $u$ is not a common inneighbor of the vertices in $Y$.
Thus, there exists a vertex $y \in Y$ such that $u \geq y$ in $P$.
Likewise, $u \not\in Y$ so there exists a vertex $x \in X$ such that $x \geq u$.
By transitivity, $x \geq y$, contradicting the fact that there is an arc from $x$ to $y$
in $K$. Therefore, $Y = \overline{X}$. 

Hence the elements of $G$ are determined by their first coordinates.
It is easy to see that these first coordinates are ideals of $P$,
that is, elements of $\mathcal{I}(P)$. Indeed, if $(X,\overline{X}) \in V(G)$
and $x < y \in X$, then $(y,x)$ is not an arc of $G$, hence $x \not \in \overline{X}$.
So it only remains to show that adjacency in $G$ corresponds to adjacency in
$\nd(\mathcal{I}(P))$. By definition $((X,\overline{X}),(Y, \overline{Y}))$ is an
arc of $G$ if and only if $\overline{X}$ intersects $Y$, that is, $X \not \supseteq Y$,
which is equivalent to $(X,Y) \in A(\nd(\mathcal{I}(P)))$.
\end{proof}
\begin{corollary} \label{drtond}
For any poset $P$ and any integer $k$, $\delta_R^k(\nd(P))$ retracts to 
$\nd(\mathcal{I}^k(P))$.
\end{corollary}

\subsection{Symmetric restrictions}
Let $[G]$ denote the symmetric restriction of a digraph $G$,
that is, the subgraph spanned by its symmetric arcs.
\begin{lemma} \label{ndtokw}
For any poset $P$, $[\nd(P)]$ retracts to a complete subgraph
with cardinality equal to the width of $P$.
\end{lemma}
\begin{proof}
Any antichain in $P$ is a complete subgraph of $[\nd(P)]$.
Let $K = [A]$, where $A$ is a maximum antichain in $P$. Then
$K$ is a complete subgraph of $[\nd(P)]$ of cardinality equal 
to the width $w$ of $P$. By Dilworth's theorem, there exists a 
chain partition $\{C_1, \ldots, C_w\}$ of $P$.
Each chain $C_i$ is an independent set in $[\nd(P)]$.
Therefore, the map $\rho: [\nd(P)] \rightarrow K$ mapping each
$C_i$ to its intersection with $A$ is a retraction of $[\nd(P)]$
to $K$.
\end{proof}

To summarize the proof of Theorem~\ref{rtwz},
the existence of a $n$-colouring of $\delta^k(G)$
is equivalent to the existence of a homomorphism
of $G$ to $\delta_R^k(K_n)$ by Corollary~\ref{coltodr}.
Since $K_n = \nd(\overline{K}_n)$,
this is equivalent to
the existence of a homomorphism
of $G$ to $\nd(\mathcal{I}^k(\overline{K}_n))$
by Corollary~\ref{drtond}.
When $G$ is undirected, its edges must be
mapped to those of $[\nd(\mathcal{I}^k(\overline{K}_n))]$ 
which retracts to $K_{b(n,k)}$ by Lemma~\ref{ndtokw}.
This concludes the proof of Theorem~\ref{rtwz}.

\section{Further comments} For any poset $P$, the chromatic
number of $\nd(P)$ is equal to the number of elements in $P$,
since for any two elements of $P$ are joined by an arc in
at least one direction. In particular,  the number of elements
of $\mathcal{I}^k(\overline{K}_n)$ is equal to the maximum possible
chromatic number of a digraph $G$ with the property that
$\chi(\delta^k(G)) = n$. Note that $G = \nd(\mathcal{I}^k(\overline{K}_n))$
achieves the bound, as well as any spanning tournament in 
$\mathcal{I}^k(\overline{K}_n)$. In particular, the maximum possible
chromatic number of a digraph $G$ with the property that
$\chi(\delta^k(G)) = n$ is achieved with $G$ being a transitive tournament.
Specializing to $k = 2$, this implies that the $n$-th Dedekind number
is equal to the maximum cardinality of a transitive tournament $T$
such that $\chi(\delta^2(T)) = 2$.

\end{document}